\documentclass[11pt,reqno,a4paper]{amsart}

\usepackage{array}
\usepackage{enumitem}
\usepackage{mathtools}
\usepackage{pgfplots}
\usepackage{soul}
\usepackage{bm} 
\usepackage{mathrsfs} 
\usepackage{pifont} 
\usepackage{graphicx} 
\usepackage{empheq} 

\usepackage{ulem} 

\usepackage{amsmath,amssymb,amsfonts}
\usepackage{tikz}
\usepackage{hyperref}
\usepackage{appendix}
\usepackage{scalerel}
\usepackage{graphicx}
\usepackage{subcaption}
\usepackage{amsthm}
\usepackage{multirow}
\usepackage{pdflscape}
\usepackage{afterpage}
\usepackage{capt-of} 
\usepackage{lipsum}
\usepackage{booktabs}
\usepackage{siunitx}
\usepackage{esvect}

\tikzset{decorated arrows/.style={
    postaction={
        decorate,
        decoration={
            markings,
            mark=between positions 0 and 1 step 15mm with {\arrow[black]{stealth};}
            }
        },
    }
}
 
\tikzset{decorated arrows2/.style={
    postaction={
        decorate,
        decoration={
            markings,
            mark=at position 15mm with {\arrow[black]{stealth};}
            }
        },
    }
}

\usepgfplotslibrary{colormaps}
\usetikzlibrary{arrows}

\makeatletter
\def\set@curr@file#1{%
  \begingroup
    \escapechar\m@ne
    \xdef\@curr@file{\expandafter\string\csname #1\endcsname}%
  \endgroup
}
\def\quote@name#1{"\quote@@name#1\@gobble""}
\def\quote@@name#1"{#1\quote@@name}
\def\unquote@name#1{\quote@@name#1\@gobble"}
\makeatother
\usepackage{graphics}

\theoremstyle{plain}
\newtheorem{thm}{Theorem}
\newtheorem{lem}{Lemma}

 \theoremstyle{definition}

\theoremstyle{remark}

\theoremstyle{plain}

\newcommand{\RR}{\mathbb{R}}

\newcommand{\dpt}{\displaystyle}

\newcommand{\RN}[1]{%
  \textup{\uppercase\expandafter{\romannumeral#1}}%
}

\author[J. P. S. M. de Carvalho]{Jo\~ao P. S. Maur\'icio de Carvalho\\
\\
\MakeLowercase{jp.carvalho@upt.pt} \smallskip \\ \medskip {\bf ORCID:} 0000-0001-7709-1631 \\
\\ 
Research on Economics, Management and Information Technologies, Prince Henry
Portucalense University \\ \medskip Rua Dr. Ant\'onio Bernardino de Almeida 541, Porto 4200-072, Portugal
}

\begin{document}




\title[Socioepidemiological dynamics]
{{\it Could society itself spiral into a Lorenz-like chaos \\ when facing an epidemic threat?}}

\date{\today}  
 
\begin{abstract}
Understanding how societies react to epidemic threats requires more than tracking infection curves. Public perception, collective memory and behavioural adaptation interact through feedback loops that can amplify or suppress the spread of fear, vigilance and precaution. In this work we reinterpret the classical Lorenz system in a socioepidemic context, governed by nonlinear interactions between perceived infection, social transmission behaviour and memory of past risk. We provide a qualitative analysis of the model and show that small fluctuations in perception or behaviour can trigger transitions between stable, oscillatory and chaotic collective responses. These results suggest that social reactions to epidemics may evolve according to intrinsic dynamical rules, generating complex patterns of vigilance, fatigue and renewed concern that mirror the irregular rhythms observed in real outbreaks. Our findings highlight the importance of incorporating behavioural feedbacks into epidemic modeling and reveal how chaotic dynamics may arise not only from pathogens but from society itself.
\end{abstract}

\maketitle


\section{Introduction and background}

\noindent Mathematical models have long been crucial tools for understanding the spread and decline of infectious diseases among populations \cite{Brauer2017}. From the early works of Kermack and McKendrick in 1927 and 1932 \cite{KermackMcKendrick1927, KermackMcKendrick1932}, who developed the first compartmental model for infectious disease transmission (SIR model), to contemporary formulations, the evolution of epidemiological modeling has been a continuous effort to capture the biological transmission of diseases among individuals \cite{BrauerChavez2019, Hethcote2000, Cobey2020}.

\smallskip

\noindent However, as recent global crises such as COVID-19 have demonstrated, epidemics are not only biological events, but also deeply social ones \cite{Dougan2020}. Fear, often amplified by media coverage and social media, can spread faster than the virus itself, shaping how individuals assess danger, comply with restrictions and interact with each other. Memory, both personal and collective, can vanish more quickly than the emergence of biological immunity: as the urgency of past warnings fades, institutions return to routine and societies can become risk-insensitive and trivialize the problem. At the same time, the collective mindset oscillates between alarmism and apathy, alternating between periods of vigilance and fatigue \cite{Yuan2022}, discipline and denial \cite{Valganon2024}. Public concern, adherence to preventive measures and the persistent memory of imminent risk rise and fall at a rate that seems to be governed by its own laws. Can these recurring cycles of vigilance and negligence be understood as the result of more complex dynamics, similar to those that govern certain chaotic systems?

\smallskip

\noindent The theory behind dynamical systems provides a compelling lens through which to approach this question. Since Lorenz's discovery in 1963 \cite{Lorenz1963}, it has been known that simple systems of feedback equations generate complex patterns of behaviour that are unpredictable but internally consistent. The Lorenz attractor, born from equations designed to describe convection in the atmosphere, has become a symbol of how deterministic mechanisms can produce chaotic results \cite{Lorenz1963}. Its characteristic loops of rising and falling, approaching and escaping, seem almost like a metaphor for the cycles of chaotic social reaction experienced during epidemics. Could a dynamic similar to Lorenz's also emerge in the social structure itself, influencing how populations percept the risk of contagion, remember past experiences with diseases and adjust their preventive behaviours? {\it Could society itself spiral into a Lorenz-like chaos when facing an epidemic threat?} If the public perception of the epidemic danger shapes behaviour (e.g., wearing masks, social distancing or seeking vaccination) and these behaviours, in turn, change the perception of risk (reducing or amplifying it), then the consequent feedback loops can generate oscillations, instabilities or even chaotic social behaviour. Exploring this analogy between physical turbulence and social turbulence reveals a new frontier in epidemic modeling.

\smallskip

\noindent For all these reasons, it is possible that the reaction of society to contagion may not only be a passive reflection of infection curves, but also an autonomous dynamic process, with its own attractors, cycles and bifurcations, that is, a complex system in which collective behaviour evolves according to nonlinear laws similar to those of the virus itself.

\smallskip

\noindent Although the application of Lorenz-type systems to the dynamics of social perception during epidemics remains largely unexplored, several authors have extended his model into other fields, such as chemistry, ecology or population dynamics.



\medskip

\noindent \textbf{Literature.} In 2024, Shi {\it et al.}~\cite{Shi2024} introduced and tested new three- and five-parameter performance equations for fitting rotated and shifted Lorenz curves of plant size distributions, proving that the generalised equation provides superior and robust fits across a wide range of asymmetric size distributions. In 2025, Singh \cite{Singh2025} modified the classical Lorenz system to include the effects of an ultra-high intensity electromagnetic field, proposing a non-autonomous Lorenz-type Laser Field model with time-dependent parameters and showing, through stability analysis, Lyapunov exponents, bifurcation diagrams and basins of attraction, that the accelerated dynamics of electrons exhibit chaotic behaviour, including regime transitions, intermittent chaos and strong attractor deformations. Also in 2025, Yan {\it et al.}~\cite{Yan2025} developed a new chaotic system similar to Lorenz with cubic nonlinearity and demonstrated that its transient chaos can be exploited for detecting weak signals in bearing vibration data, significantly improving fault diagnosis in noisy environments.

\smallskip

\noindent Many other applications of the Lorenz system can be found in the references referenced in these studies. However, as far as we know, there are no studies that address applications of the Lorenz system to social behaviour.

\medskip

\noindent \textbf{Goals and novelty.} This work extends the classical Lorenz-type dynamic structure

\smallskip

\begin{center}
$\dot{X} = - \sigma X + \sigma Y\, , \quad \dot{Y} = - XZ + rX - Y\, , \quad \dot{Z} = XY - b Z$   \qquad  \cite{Lorenz1963}
\end{center}

\smallskip

\noindent to a socioepidemiological context, linking nonlinear dynamics to the collective social response to epidemics. Although the Lorenz system has been widely studied and applied in physics and other sciences, its application to social perception and behavioural adaptation during epidemics remains unexplored. We reinterpret Lorenz's equations using variables that represent the intensity of social transmission, perceived infection and collective memory, capturing how feedback between perception, memory and behaviour can lead to stable, oscillatory or chaotic regimes. A qualitative analysis is performed to characterise these transitions and provide new insights into the nonlinear nature of collective perception and social behaviour in epidemic contexts.


\medskip

\noindent \textbf{Structure.} In this article, we analyze a Lorenz-type socioepidemiological model that relates social transmission, perceived infection and collective memory during epidemics. Section \ref{model_model} presents the model, parameters and the main result, while in Section \ref{SECTION_3} we prove the theorem of the main result through a qualitative analysis of the system. Section \ref{sec_4} provides an interpretation of the different regimes present in the results of this paper. In Section \ref{str_str_sec}, we provide a social interpretation of the strange attractor that derives from the variation of $r_0$ . Finally, Section \ref{DISCUSSION} provides a discussion, where we draw an analogy between $r_0$ and the basic reproduction number $\mathcal{R}_0$, the limitations of the model and future work.



\section{Model dynamics and main result} \label{model_model} 

\noindent The proposed model organizes the socioepidemiological dynamics into three variables:

\medskip

\begin{itemize}
\item[$\boldsymbol{T(t)}${\bf :}] represents the intensity of social transmission, {\it i.e.} the general propensity of the population toward behaviours that help the spread of infection, such as mobility, social interaction or relaxation of preventive measures;

\medskip

\item[$\boldsymbol{I(t)}${\bf :}] expresses the perception of infection intensity, corresponding to the collective perception of epidemic risk (media attention, public concern or awareness of infection status);

\medskip

\item[$\boldsymbol{M(t)}${\bf :}] denotes social memory, which accumulates and holds the adherence of the population to protection norms and practices (e.g., precautionary habits, compliance with health measures or institutional responses).
\end{itemize}

\medskip

\noindent With respect to our model, we define ${{\Lambda} \subset \left\{ (\sigma, r_0, \beta) \in (\mathbb{R}^+)^3 \right\}}$ as the set of parameters. The parameter $\boldsymbol{\sigma}$ is the behavioural adjustment rate, governing how quickly the level of social transmission $T$ reacts to changes in perceived infection $I$. When $\sigma$ is large, the population adjusts almost immediately, tightening precautions when risk perception rises and relaxing them when the perceived threat diminishes. Accordingly, the first equation $\dot{T}$, captures this behavioural feedback, where social behaviour continuously aligns with perceived epidemic pressure. The second equation $\dot{I}$, governs the dynamics of perceived infection. The term $T(r_0 - M)$ expresses that perceived infection grows when social transmission is intense ($T$ is high) and collective response is weak ($M$ is low). The parameter $\boldsymbol{r_0}$ represents the infection potential or maximum susceptibility of the system, modulating how strongly behavioural exposure translates into perceived epidemic pressure. The negative term $-I$ models the natural fading of risk perception, as attention and concern decline in the absence of reinforcing events. The third equation $\dot{M}$, describes the evolution of social memory. The memory variable increases when both exposure ($T$) and perceived infection ($I$) are high (since direct experience of risk promotes learning and adherence to protective measures). However, this memory gradually decays at rate $\boldsymbol{\beta}$, reflecting collective forgetting, fatigue or resistance to sustained preventive efforts.

\smallskip

\noindent These coupled mechanisms create a web of reciprocal feedbacks: higher perceived infection reduces social transmission, diminished transmission lowers perceived risk and weakened memory allows transmission to rise again. This interplay can generate oscillatory or even chaotic dynamics, echoing the recurrent cycles of alarm, adaptation and relaxation often observed in real epidemic contexts. A schematic overview of these interactions is shown in Figure \ref{diagram_lorenz_1}.

\smallskip

\noindent The dynamical system that captures these interacting processes consists of the following nonlinear system of ODE in the variables $T$, $I$ and $M$, which vary in time $t\in \RR_0^+$:


\medskip

\begin{equation}
\label{modeloSIR}
\begin{array}{lcl}
\dot{X} = \mathcal{F}(X) \quad \Leftrightarrow \quad
\begin{cases}
&\dot{T} = \sigma\,(I - T) \\
\\
&\dot{I} = T\,(r_0 - M) - I \\
\\
& \dot{M} = TI - \beta M
\end{cases}
\end{array}
\end{equation}

\medskip

\noindent where 

\smallskip

$$
\begin{array}{rcl}
\medskip
\smallskip
\smallskip
X(t) &=& \left( T(t), I(t), M(t) \right), \\
\medskip
X(t_0) &\coloneqq& X_0 \,\,\, \text{is the initial condition}, \\
\dot{X} &=& \left(\dot{T}, \dot{I}, \dot{M} \right) \,\,\, = \,\,\, \dpt \left(\frac{\mathrm{d}T}{\mathrm{d}t},\frac{\mathrm{d}I}{\mathrm{d}t}, \frac{\mathrm{d}M}{\mathrm{d}t} \right) .
\end{array}
$$

\medskip

\noindent The vector field associated to \eqref{modeloSIR} will be denoted by $\mathcal{F}$ and its flow is $\phi \left(t, (T_0, I_0, M_0) \right)$, $t \in \mathbb{R}_0^+$ and $(T_0, I_0, M_0) \in (\mathbb{R}_0^+)^3$.

\begin{figure}[ht!]
\includegraphics[width=1.48 \textwidth]{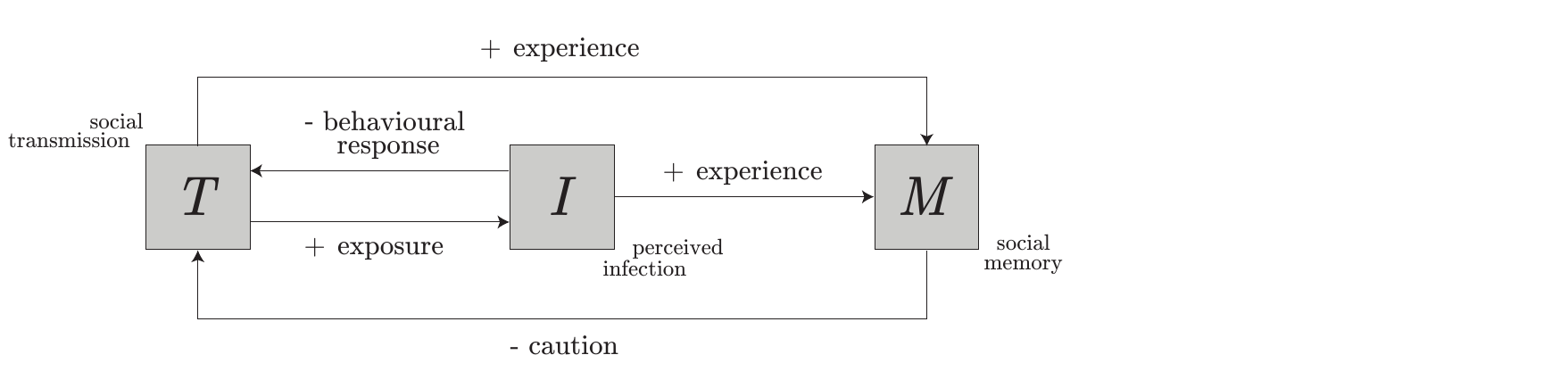}
\caption{Schematic representation of the socioepidemiological model \eqref{modeloSIR}. The variables $T$, $I$ and $M$ denote social transmission, perceived infection and social memory, respectively. Transmission increases perceived infection through exposure $(+)$, while perceived infection reduces transmission through behavioural response $(-)$. Both $T$ and $I$ contribute positively to the construction of social memory $(+)$, which later acts on transmission, promoting caution and adherence to norms $(-)$.}
\label{diagram_lorenz_1}
\end{figure} 

\noindent \textbf{Motivation.} The Lorenz system \cite{Lorenz1963}, known for its nonlinear and sensitive dynamics, inspired us to reinterpret socioepidemiological behaviour. We thus propose the Lorenz-type model \eqref{modeloSIR}, where:

\begin{itemize}
\item The variables $T$, $I$ and $M$ denote the intensity of social transmission, risk perception and collective memory, respectively;

\medskip

\item Nonlinear feedbacks among these variables capture oscillatory and complex social-epidemic dynamics;

\medskip

\item Following the Lorenz attractor approach, we examine invariant regions and the asymptotic behaviour of \eqref{modeloSIR}.
\end{itemize}

\medskip

\noindent \textbf{Hypotheses.} 
The proposed Lorenz-type socioepidemiological model \eqref{modeloSIR} follows general assumptions that ensure mathematical consistency and interpretability of its dynamics. Following standard approaches in dynamic systems theory, we adopt the following conditions:

\medskip

\begin{itemize}
\item[{\bf (H1)}] All parameters are positive;

\medskip

\item[{\bf (H2)}] The vector field defining system \eqref{modeloSIR} is continuously differentiable and therefore locally Lipschitz in $(T,I,M)$, ensuring local existence and uniqueness of solutions for any initial condition.
\end{itemize}

\medskip

\noindent Under these assumptions, the qualitative behaviour of the trajectories can be rigorously analyzed. 
In particular, the next result establishes that the system is globally well defined and dissipative, admitting a bounded positively invariant region where all trajectories eventually remain.

\begin{lem}[Dissipativity and global existence] \label{dissipative}
Assume {\bf (H1)} and {\bf (H2)}. 
Then, for any initial condition, the solution $(T(t),I(t),M(t))$ exists for all $t \ge 0$ and remains bounded. 
In particular, there exists a bounded closed ball in $\mathbb{R}^3$ that is positively invariant and absorbing.
\end{lem}

\begin{proof}
Since the vector field is continuously differentiable, it is locally Lipschitz in $(T,I,M)$. 
By the Picard--Lindel\"{o}f theorem \cite{Picard1890}, there exists a unique maximal solution $(T,I,M)$ defined on an interval $[0,t_{\max})$.

\smallskip

\noindent We now define the following Lyapunov function

$$
V(T,I,M) = T^2 + I^2 + (M - a)^2, \qquad \text{where} \quad a = \sigma + r_0.
$$

\medskip

\noindent Differentiating $V$ along the trajectories of system \eqref{modeloSIR} gives

$$
\dot{V} = 2T\dot{T} + 2I\dot{I} + 2(M - a)\dot{M}
         = -2\sigma T^2 - 2I^2 - 2\beta M(M - a).
$$

\medskip

\noindent Since $M(M - a) = (M - a)^2 + a(M - a)$, we can write

$$
\dot{V} = -2\sigma T^2 - 2I^2 - 2\beta (M - a)^2 - 2\beta a (M - a).
$$

\medskip

\noindent Using the inequality $-2uv \le u^2 + v^2$ with $u = \sqrt{\beta}(M - a)$ and $v = \sqrt{\beta}a$, we obtain

$$
-2\beta a(M - a) \le \beta (M - a)^2 + \beta a^2,
$$

\medskip

\noindent hence

$$
\dot{V} \le -2\sigma T^2 - 2I^2 - \beta (M - a)^2 + \beta a^2.
$$

\medskip

\noindent Let $m := \min\{2\sigma,\,2,\,\beta\}$ and note that $V = T^2 + I^2 + (M - a)^2$. Then,

$$
\dot{V} \le -mV + c, \qquad \text{where} \quad c = \beta a^2.
$$

\medskip

\noindent Now, integrating the inequality yields

$$
V(t) \le V(0)e^{-mt} + \frac{c}{m}\bigl(1 - e^{-mt}\bigr)
      \quad \Rightarrow \quad
      \sup_{t \ge 0} V(t) \le \max\Bigl\{V(0),\,\frac{c}{m}\Bigr\} < \infty.
$$

\medskip

\noindent Thus, $(T(t),I(t),M(t))$ remains bounded for all $t \ge 0$, which implies that no finite-time blow-up occurs and $t_{\max} = \infty$ (global existence).

\medskip

\noindent Since $\dot{V} < 0$ whenever $V > c/m$, the closed ball

$$
\mathcal{B} := \Bigl\{(T,I,M) : V(T,I,M) \le R^2 \Bigr\}, \qquad \text{where} \quad 
R^2 = \dfrac{c}{m} = \dfrac{\beta(\sigma + r_0)^2}{\min\{2\sigma,\,2,\,\beta\}},
$$

\medskip

\noindent is positively invariant and absorbing. Every trajectory eventually enters $\mathcal{B}$ and remains within it thereafter. The Lemma is proved.
\end{proof}

\noindent Let $P_0$ be the trivial equilibrium point and $P_e^{\pm}$ be the two equilibria of system \eqref{modeloSIR}. Let

\begin{eqnarray} \label{r_H}
\nonumber r_H = \dfrac{\sigma(\sigma+\beta+3)}{\sigma-\beta-1} \, .
\end{eqnarray}

\medskip

\noindent We have the following result:

\medskip


\begin{thm} \label{THM_1}
\noindent Consider system~\eqref{modeloSIR} with parameters $(\sigma,r_0,\beta)\in\Lambda\subset(\mathbb{R}^+)^3$. Let $P_0$ be the trivial equilibrium and let $P_e^+$ and $P_e^-$ be the two non-trivial equilibria. Hence,







\medskip

\begin{itemize}
\item if \,$r_0=1$, then system~\eqref{modeloSIR} undergoes a pitchfork bifurcation;

\medskip
\smallskip

\item if \,$r_0>1$, then two symmetry-related equilibria $P_e^{\pm}$ bifurcate from $P_0$.
\end{itemize}

\medskip
\smallskip

\noindent Assuming $\sigma>\beta+1$,

\medskip
\smallskip

\begin{itemize}
\item if \,$r_0=r_H$, then each equilibrium $P_e^{\pm}$ undergoes a Hopf bifurcation.
\end{itemize}


\end{thm}


\noindent The proof of Theorem \ref{THM_1} is provided in Section \ref{SECTION_3}.

\medskip

\noindent \textbf{Remark.} Theorem~\ref{THM_1} describes local bifurcations of the equilibria but does not guarantee the existence of a strange attractor for all values of $r_0>r_H$. As in the classical Lorenz system, chaotic dynamics and Lorenz-like strange attractors may occur only for specific values of parameters. In this work, the evidence for chaotic behaviour is based on numerical simulations for particular values of the parameters (see Section~\ref{str_str_sec}).


\section{Qualitative analysis} \label{SECTION_3}

\noindent We derive the equilibria of the system and discuss their stability, as well as the bifurcations they undergo. The jacobian matrix of the vector field associated to \eqref{modeloSIR}, evaluated at a general point $P = (T,I,M) \in (\mathbb{R}_0^+)^3$, is given by:

\begin{eqnarray}
\label{jacob}
\mathcal{J}(P)=\left(\begin{array}{ccc}
- \sigma & \sigma & 0 \\ 
\\
r_0 - M & -1 & - T \\
\\
I & T & - \beta
\end{array}\right) \, .
\end{eqnarray}

\medskip

\noindent We compute the equilibria by founding the zeros from $\mathcal{F}(X^{\star}) = 0$ (see \eqref{modeloSIR}), {\it i.e.}

\begin{equation}
\begin{array}{lcl}
\begin{cases}
& \sigma \left(I^{\star} - T^{\star} \right) = 0\\
\\
&T^{\star} \left(r_0 - M^{\star} \right) - I^{\star} = 0  \\
\\
& T^{\star} I^{\star} - \beta M^{\star} = 0 \, ,
\end{cases} .
\end{array}
\label{1st_step}
\end{equation} 

\medskip

\noindent where $\left(T^{\star}, I^{\star}, M^{\star} \right)$ is a generic equilibrium point. From the first equation we know that:

\begin{equation}
\begin{array}{lcl}
I^{\star} = T^{\star} \, .
\end{array}
\label{2nd_step}
\end{equation} 

\medskip

\noindent If we use \eqref{2nd_step} in the second equation of \eqref{1st_step}, we have

\begin{eqnarray*}
T^{\star} \left(r_0 - M^{\star} \right) - T^{\star} = 0 \quad &\Leftrightarrow& \quad T^{\star} \left(r_0 - M^{\star} - 1 \right) = 0 \, .
\label{3nd_step}
\end{eqnarray*} 

\medskip

\noindent Now we face two cases:

\medskip

\begin{enumerate}
\item[1.] $T^{\star} = 0$, and then we have the trivial equilibrium point $P_0 = (0,0,0)$, which represents a socially inactive state in which there is no perception of infection, no social behaviour related to transmission and no collective memory or response. This corresponds to the absence of epidemic perception or social reaction among the population;

\medskip
\smallskip

\item[2.] $M^{\star} = r_0 -1$. 

\smallskip

\noindent Substituting this expression into the third equation of \eqref{1st_step}, we get:

$$
{T^{\star}}^2 = \beta(r_0 - 1) \quad \Leftrightarrow \quad T^{\star} = \pm \sqrt{\beta \left(r_0-1\right)} \, ,
$$

\medskip

\noindent and then we have two non-trivial symmetric equilibria 

\begin{eqnarray*}
\label{trivial_eq}
P_e^{\pm} = \left(T_e^{\pm}, I_e^{\pm}, M_e\right) = \left(\pm \alpha, \pm \alpha, r_0-1 \right) \, ,
\end{eqnarray*}

\medskip

\noindent where $\alpha = \sqrt{\beta \left(r_0 - 1\right)}$ and which exist only for $r_0 > 1$.

\medskip
\smallskip

\noindent Each point $P_e^{\pm}$ represents a socially active endemic state, where the population holds a persistent level of infection awareness ($I > 0$), transmission-related social behaviour ($T > 0$) and collective memory ($M > 0$). 

\smallskip

\noindent The symmetry of the signal means there are two possible ways the social response could go:

\medskip

\begin{itemize}
\item $P_e^{+}$ corresponds to a scenario in which social perception and behavioural response are aligned with the risk of infection (collective adherence to protective measures, sustained awareness and behavioural synchronization that neutralize transmission);

\medskip

\item $P_e^{-}$ corresponds to the opposite orientation, in which collective behaviour reinforces risk (denialism, fatigue or coordinated behaviour that favors transmission despite awareness).
\end{itemize}

\medskip

\noindent These two equilibria are mirror images under the $\mathbb{Z}_2$ symmetry

\smallskip

\begin{center}
$(T^{\star},I^{\star},M^{\star})\mapsto(-T^{\star},-I^{\star},M^{\star})$\,\, ,
\end{center}

\smallskip

\noindent representing alternative stable social regimes that can emerge after the loss of stability of the inactive state $P_0$ when $r_0$ crosses the critical value 1.
\end{enumerate}

\subsection{Trivial equilibrium stability}

\noindent We have the following result regarding to the stability of $P_0$:

\medskip

\begin{lem} \label{P0_stab}
With respect to system~\eqref{modeloSIR}, the following holds for $P_0$:

\medskip

\begin{itemize}
\item If \,$r_0 < 1$, then $P_0$ is a hyperbolic sink and a global attractor;

\medskip
\smallskip

\item If \,$r_0 > 1$, then $P_0$ is a saddle.
\end{itemize}
\end{lem}

\smallskip

\begin{proof}
\noindent At $P_0$, the matrix \eqref{jacob} takes the form

\begin{eqnarray*}
\label{jacob_P0}
\mathcal{J}(P_0)=\left(\begin{array}{ccc}
-\sigma  & \sigma & 0  \\ 
\\
r_0 & -1 & 0 \\
\\
0 & 0 & -\beta
\end{array}\right) \, .
\end{eqnarray*}

\medskip

\noindent The matrix $\mathcal{J}(P_0)$ has a block-diagonal structure with a $2\times2$ block on $(T,I)$ and a $1\times1$ block on $M$. Then we have

\medskip

$$
\det\!\big(\mathcal{J}(P_0)-\lambda I\big)
=\det\!\begin{pmatrix}
-\sigma-\lambda & \sigma\\
r_0 & -1-\lambda
\end{pmatrix}\,(-\beta-\lambda) \, ,
$$

\medskip

\noindent where $I$ is the identity matrix. Therefore $-\beta < 0$ is an eigenvalue and the remaining two eigenvalues are roots of

$$
p(\lambda)=\det\!\begin{pmatrix}
-\sigma-\lambda & \sigma\\
r_0 & -1-\lambda
\end{pmatrix}
=(\lambda+\sigma)(\lambda+1)-\sigma r_0
=\lambda^2+(\sigma+1)\lambda+\sigma(1-r_0).
$$

\medskip

\noindent Let $\lambda_{1,2}$ denote the roots of $p$. By Vi\`et's formulas,

$$
\lambda_1+\lambda_2=-(\sigma+1)<0,
\qquad
\lambda_1\lambda_2=\sigma(1-r_0).
$$

\medskip

\noindent Moreover, the discriminant

$$
\Delta=(\sigma+1)^2-4\sigma(1-r_0)=(\sigma-1)^2+4\sigma r_0>0 \, ,
$$

\medskip

\noindent hence $\lambda_{1,2}\in\mathbb{R}$ for all $\sigma,r_0>0$.

\medskip

\noindent If \,$r_0<1$, then \,$\lambda_1\lambda_2=\sigma(1-r_0)>0$ \,and \,$\lambda_1+\lambda_2<0$, hence \,$\lambda_1<0$ and $\lambda_2<0$. Together with \,$\lambda_3=-\beta<0$, all eigenvalues of $\mathcal{J}(P_0)$ have strictly negative real parts and $P_0$ is a hyperbolic sink (stable). In particular, since system \eqref{modeloSIR} is dissipative (from Lemma \ref{dissipative}), and $P_0$ is the unique equilibrium point, then $P_0$ is a global attractor for $r_0 < 1$.

\medskip

\noindent If \,$r_0>1$, then \,$\lambda_1\lambda_2=\sigma(1-r_0)<0$, so $p$ has one positive and one negative root. Since $\lambda_3 = -\beta < 0$, then $\mathcal{J}(P_0)$ has a positive eigenvalue and two others negative eigenvalues. Therefore $P_0$ is a saddle with one-dimensional unstable manifold.

\medskip

\noindent Hence the Lemma is proved.

\end{proof}

\subsection{Non-trivial equilibria stability}

\noindent For $\sigma > \beta + 1$, we have $r_H > 1$. Hence, the interval $(1,r_H)$ is non-empty. Concerning the two non-trivial equilibria, we establish the following result on their local stability:

\begin{lem} \label{Pe_stab}
\noindent With respect to system \eqref{modeloSIR}, the following holds for the non-trivial equilibria:

\medskip

\begin{itemize}
\item If \,$1 < r_0 < r_H$, then $P_e^{\pm}$ are stable;

\medskip
\smallskip

\item If \,$r_H < r_0$, then $P_e^{\pm}$ are unstable.
\end{itemize}
\end{lem}

\smallskip

\begin{proof}
\noindent Evaluating $\mathcal{J}(P)$ (see \eqref{jacob}) at $P_e^{\pm}$ we have

\begin{eqnarray}
\label{jacob_Pe}
\mathcal{J}(P_e^{\pm})=\left(\begin{array}{ccc}
-\sigma & \ \sigma & 0\\
\\
1 & -1 & \mp\,\alpha\\
\\
\pm\,\alpha & \ \pm\,\alpha & -\beta
\end{array}\right) \, .
\end{eqnarray}

\medskip

\noindent A direct computation of $\det\left(\lambda I- J(P_e^{\pm})\right)$ leads to the following characteristic
polynomial

\begin{equation*}\label{eq:charpoly-Ppm}
\chi(\lambda)=\lambda^3 + a_1\,\lambda^2 + a_2\,\lambda + a_3,
\end{equation*}

\medskip

\noindent where $a_1=\sigma+\beta+1$, $a_2=\beta(\sigma+r_0)$ and $a_3=2\beta\sigma(r_0-1)$. According to the Routh--Hurwitz criterion for cubic polynomials, all roots of $\chi$ have negative real parts if and only if

$$
a_1>0,\quad a_2>0,\quad a_3>0 \quad \text{and}\quad a_1a_2>a_3.
$$

\medskip

\noindent For $r_0>1$ and $\sigma,\beta \in \mathbb{R}^+$, the first three inequalities are trivial: $a_1=\sigma+\beta+1>0$, $a_2=\beta(\sigma+r_0)>0$ and $a_3=2\beta\sigma(r_0-1)>0$. Hence the only non-trivial condition is precisely

$$
(\sigma+\beta+1)\,\beta(\sigma+r_0)\ >\ 2\beta\sigma\,(r_0-1) \, .
$$

\medskip

\noindent We can get this expression in order to $r_0$, {\it i.e.}

\begin{eqnarray*}
\nonumber && (\sigma+\beta+1)\beta(\sigma+r_0) > 2\beta\sigma(r_0-1) \\
\nonumber \\
\nonumber &\Leftrightarrow& (\sigma+\beta+1)(\sigma+r_0) > 2\sigma(r_0-1) \\
\nonumber \\
\nonumber &\Leftrightarrow& \sigma^2 + r_0 \sigma + \beta \sigma + r_0 \beta + \sigma + r_0  > 2 r_0 \sigma - 2 \sigma \\
\nonumber \\
\nonumber &\Leftrightarrow& \sigma^2 - r_0 \sigma + \beta \sigma + r_0 (\beta + 1) + \sigma + 2 \sigma > 0 \\
\nonumber \\
\nonumber &\Leftrightarrow& \sigma^2 + \beta \sigma + 3 \sigma + r_0 (\beta + 1 - \sigma) > 0 \\
\nonumber \\
\nonumber &\Leftrightarrow& r_0(\sigma - \beta - 1) < \sigma^2 + 3 \sigma + \beta \sigma \\
\nonumber \\
&\Leftrightarrow& r_0 < \dfrac{\sigma(\sigma + \beta + 3)}{\sigma - \beta - 1} = r_H \, . \label{hopf_bif}
\end{eqnarray*}

\medskip

\noindent Thus, if \,$\sigma>\beta+1$, then \,$r_0 < r_H$ \,and \,$P_e^{\pm}$ are both stable. Otherwise they are unstable. The Lemma is proved.
\end{proof}

\noindent The simulations in Figure \ref{phase_spaces} illustrate the dynamics described in Lemma \ref{P0_stab} and Lemma \ref{Pe_stab}.

\begin{figure*}[ht!]
\includegraphics[width=0.461 \textwidth]{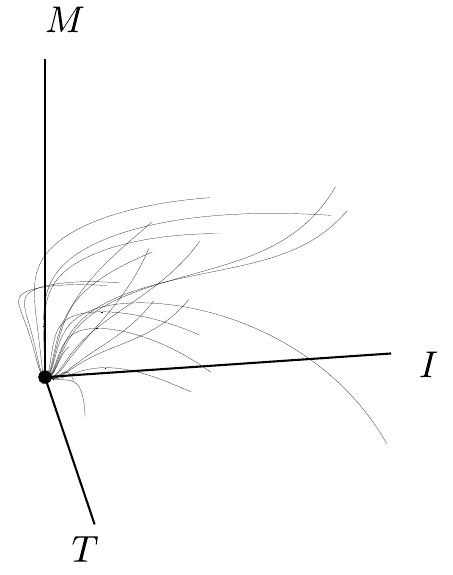} \qquad \,\,
\includegraphics[width=0.461 \textwidth]{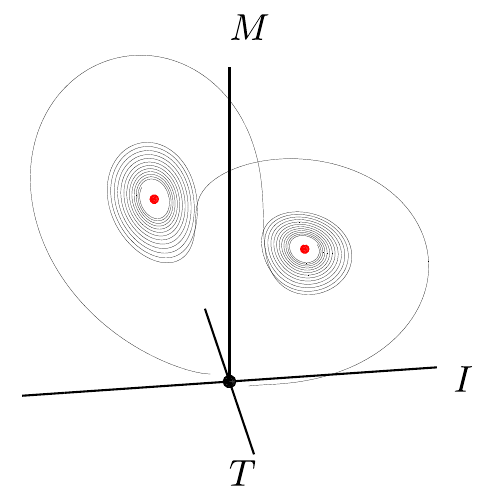}
\caption{Phase portraits of system \eqref{modeloSIR} under two parameter regimes. 
\textbf{Left:} Several initial conditions with $r_0 < 1$ 
(Parameters: $\sigma = 2$, $r_0 = 0.5$, $\beta = 1$). 
\textbf{Right:} Initial conditions near the origin with $1 < r_0 < r_H$ 
(Parameters: $\sigma = 10$, $r_0 = 20$, $\beta = 2.7$). 
Black dots denote the trivial equilibrium point $P_0$ (a sink on the left and a saddle on the right), 
while red dots denote the non-trivial equilibria $P_e^{\pm}$ (stable foci).}
\label{phase_spaces}
\end{figure*}


\subsection{Pitchfork bifurcation} \label{Pitchfork_sec}

\medskip

\noindent With respect to system~\eqref{modeloSIR}, the equilibrium $P_0=(0,0,0)$ exists for all $r_0>0$, and for $r_0>1$ there appear two non-trivial equilibria

$$
P_e^{\pm}=\big(\pm\sqrt{\beta(r_0-1)},\,\pm\sqrt{\beta(r_0-1)},\,r_0-1\big),
$$

\medskip

\noindent which merge into $P_0$ as $r_0 \to 1$.

\smallskip

\noindent The vector field is $\mathbb{Z}_2$-equivariant under $(T,I,M) \mapsto(-T,-I,M)$, so $P_e^-$ is the image of $P_e^+$ by the symmetry. From Lemma \ref{P0_stab} we know that $P_0$ is a hyperbolic sink for \,$r_0<1$ and a saddle for \,$r_0>1$. Moreover, no other equilibria exist for $r_0 \leq 1$. However, for $r_0>1$ the two symmetry-related equilibria bifurcate from $P_0$ with amplitude

$$
\alpha(r_0)=\sqrt{\beta(r_0-1)}\sim C\,\sqrt{r_0-1} \quad \, \text{as }\,\,\,r_0 \to 1^+ \, ,
$$

\medskip

\noindent where $C = \sqrt{\beta} > 0$. This scaling and the change of stability are those of a supercritical {\it pitchfork}. Hence, system~\eqref{modeloSIR} undergoes a ({\it supercritical}) {\it pitchfork bifurcation} at $r_0=1$, from which the equilibria $P_e^{\pm}$ emerge.

\subsection{Hopf bifurcation}

\medskip

\noindent We know from Lemma \ref{Pe_stab} that the characteristic polynomial of $\mathcal{J}(P_e^{\pm})$ (see \eqref{jacob_Pe}) is

$$
\chi(\lambda)=\lambda^3+a_1\lambda^2+a_2\lambda+a_3 \,,
$$

\medskip

\noindent where \,$a_1=\sigma+\beta+1$, \,$a_2=\beta(\sigma+r_0)$ \,and \,$a_3=2\beta\sigma(r_0-1)$. We also know that for \,$r_0>1$ we have $a_1,a_2,a_3>0$. Setting

\begin{eqnarray} \label{DELTA_DELTA}
\Delta(r_0)=a_1a_2-a_3=(\sigma+\beta+1)\,\beta(\sigma+r_0)-2\beta\sigma(r_0-1),
\end{eqnarray}

\medskip

\noindent the Routh-Hurwitz criterion shows that the stability boundary is $\Delta(r_0)=0$, {\it i.e.}

$$
r_0 = r_H = \dfrac{\sigma(\sigma+\beta+3)}{\sigma-\beta-1}\, .
$$

\medskip

\noindent At $r_0=r_H$, we look for purely imaginary roots by substituting $\lambda=i\omega$ (with $\omega\in\mathbb{R}$) into $\chi(\lambda)=0$. Expanding term by term, we get

\begin{eqnarray}
\nonumber \chi(i\omega)&=&(i\omega)^3+a_1(i\omega)^2+a_2(i\omega)+a_3 \\
\nonumber&=& (-i\omega^3) - a_1\omega^2 + i a_2\omega + a_3 \\
\nonumber&=& \left(-a_1\omega^2+a_3\right) + \left( -\omega^3+a_2\omega \right) i \, ,
\end{eqnarray}

\medskip

\noindent where \,$-a_1\omega^2+a_3$ \,is the real part and \,$-\omega^3+a_2\omega$ \,is the imaginary part. Setting $\chi(i\omega)=0$ this yields the two conditions


\begin{align}
-a_1\omega^2 + a_3 &= 0, \label{2_conditions_1}\\
-\omega^3 + a_2\omega &= 0. \label{2_conditions_2}
\end{align}

$$
-a_1\omega^2+a_3 =0 \qquad \text{and} \qquad -\omega^3+a_2\omega = 0 \, .
$$

\medskip

\noindent From \eqref{2_conditions_2} we factor $\omega$ and we get:

$$
\omega(-\omega^2+a_2)=0.
$$

\medskip

\noindent For a {\it Hopf bifurcation} we want $\omega \neq 0$. Then

$$
-\omega^2+a_2 = 0 \quad \Leftrightarrow \quad \omega^2=a_2.
$$

\medskip

\noindent Substituting this into \eqref{2_conditions_1} we have

$$
-a_1\omega^2+a_3=0
\quad \Leftrightarrow \quad
-a_1 a_2 + a_3 = 0
\quad \overset{\eqref{DELTA_DELTA}}{\Leftrightarrow} \quad
\Delta(r_H)=0 \, .
$$

\medskip

\noindent Therefore,

$$
\omega^2 = a_2(r_H) = \dfrac{a_3(r_H)}{a_1}>0\,.
$$

\medskip

\noindent For transversality, we compute

$$
\dfrac{\mathrm{d}\Delta(r_0)}{\mathrm{d}r_0} \Big|_{r_0=r_H}  =\dfrac{\mathrm{d}}{\mathrm{d}r_0}\left(a_1a_2-a_3\right)\Big|_{r_0=r_H}=a_1\,\beta-2\beta\sigma=\beta(\beta+1-\sigma),
$$

\medskip

\noindent which is strictly negative under the hypothesis $\sigma>\beta+1$. Therefore the eigenvalues cross the imaginary axis and each equilibrium $P_e^{\pm}$ undergoes a {\it Hopf bifurcation} at $r_0 = r_H$.

\noindent Hence, Theorem \ref{THM_1} is proved.



\section{Social interpretation of Theorem \ref{THM_1}} \label{sec_4}

\noindent In model \eqref{modeloSIR}, $r_0$ represents the overall infection potential or maximum susceptibility of the system. It measures the intensity with which behavioural exposure (social transmission $T$) translates into perceived epidemic pressure $I$. From a social perspective, $r_0$ can be interpreted as the degree of collective openness or vulnerability to the resurgence of the epidemic, which is affected by factors such as population mobility, media amplification and institutional preparedness. As $r_0$ varies, the system exhibits distinct qualitative regimes that correspond to different ways of collective behaviour:

\medskip

\begin{enumerate}
\item[\fbox{1.}] Low infection potential ($r_0 < 1$):
the system only admits the trivial equilibrium $P_0$, which is stable. From a social point of view, this means that the risk perceived by the population and the behavioural response remain minimal. The population maintains a low and constant level of concern, and preventive behaviour is weak or non-existent. Epidemic surveillance is not self-sustaining: any temporary alarm or increase in perception quickly fades, returning society to a stable (but complacent) baseline. This stage represents a society in which the epidemic threat is underestimated or effectively contained and no significant social oscillations are experienced.

\medskip

\item[\fbox{2.}] Critical transition ($r_0 = 1$):
the trivial equilibrium loses stability, giving rise to two stable non-trivial equilibria. In social terms, this is the point at which perception of the epidemic begins to strongly influence the behaviour of the population. People become dynamically responsive: small changes in risk perception can trigger noticeable adjustments in transmission behaviour and collective memory. This is the beginning of a ``responsive'' phase.

\medskip

\item[\fbox{3.}] Intermediate infection potential ($1 < r_0 < r_H)$:
At this stage, we are faced with the coexistence of two contrasting regimes, reflecting stable alternative states of collective behaviour: one of sustained vigilance (high perception and memory -- reduced transmission) and another of relaxed behaviour (low perception and memory -- increased transmission). The system tends towards an equilibrium that depends on its initial conditions, for example, historical and contextual factors that shape public perception at the outbreak of an epidemic. Empirically, we conclude that societies with similar epidemiological realities can react in very different ways, varying between caution and negligence.


\medskip

\item[\fbox{4.}] Infection potential and instability ($r_0 > r_H$):
As $r_0$ increases, the equilibria of the system become unstable and the dynamics evolve towards oscillatory regimes. In social contexts, this corresponds to recurring cycles of fear, adaptation and fatigue. Perception and behaviour feed into each other in a self-reinforcing cycle, leading to alternating waves of vigilance and relaxation. Society falls into a regime of collective instability, where minor perturbations or information shocks can lead to disproportionate and long-term effects on social behaviour.
\end{enumerate}

\smallskip

\noindent Figure \ref{illustration} illustrates the dynamics described in Theorem \ref{THM_1} and in the qualitative analysis of Section \ref{SECTION_3}. The strange attractor in this Figure is explained in the following Section \ref{str_str_sec}.

\smallskip
\smallskip

\begin{figure*}[ht!]
\includegraphics[width=1.1 \textwidth]{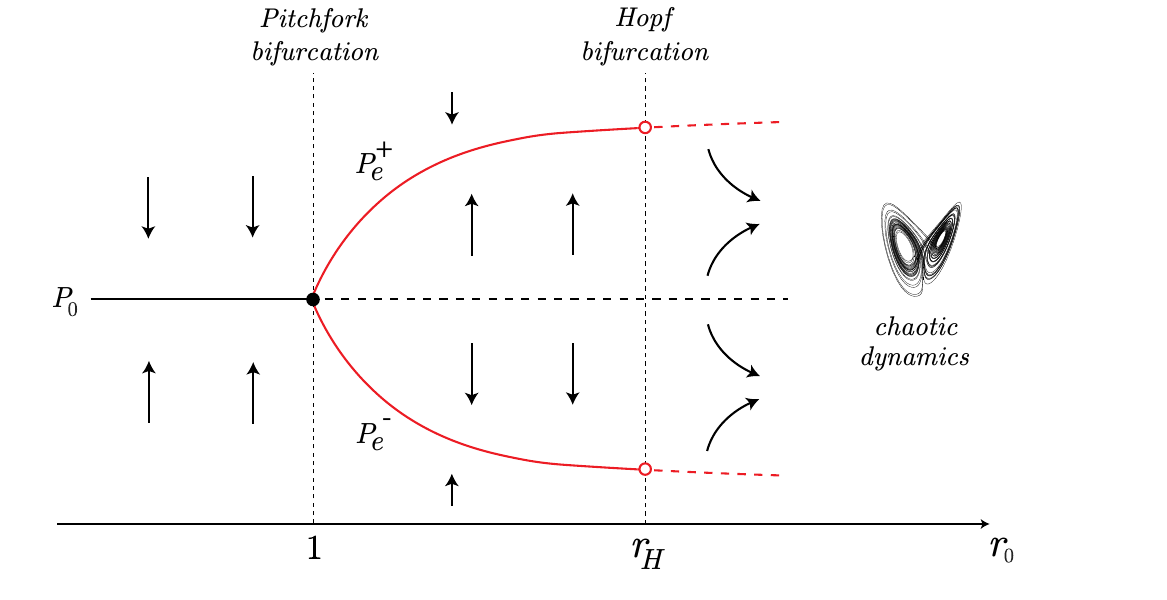}
\caption{Schematic bifurcations and stability of the equilibria for model \eqref{modeloSIR}. The trivial equilibrium point $P_0$ undergoes a {\it supercritical pitchfork bifurcation} at $r_0 = 1$ (black dot), changing its stability from a sink to a saddle and giving rise to two stable non-trivial equilibria $P_e^{\pm}$. For $r_0 = r_H$ (white dot), the equilibria $P_e^{\pm}$ undergo a {\it subcritical Hopf bifurcation} and lose their stability. For $r_0 > r_H$, the sketch of a Lorenz-like strange attractor indicates, in a purely schematic way, that for certain parameter values beyond this threshold the dynamics may become chaotic (as suggested by the numerical simulation in Figure \ref{str_atr} of Section \ref{str_str_sec}).} 
\label{illustration}
\end{figure*} 


\noindent The evolution from stability to instability as $r_0$ increases illustrates how collective behaviour, like physical-social systems, can switch from equilibrium to volatility as sensitivity and interdependence increase. This Lorenz-type socioepidemiological model provides a natural language for describing the emergence of complex social responses to epidemic threats, linking the mathematical dynamics of the model to the behavioural actions of individuals. In the following section, we show how Lorenz's strange attractor provides a natural framework for interpreting the chaotic behavioural patterns that may arise in a society facing an epidemic.


\section{The Strange Attractor as a metaphor for chaotic social dynamics} \label{str_str_sec}


\noindent When the potential for infection becomes sufficiently high (e.g., for parameter values such as $\sigma = 10$, $\beta = 8/3$ and $r_0 = 28$), numerical simulations of system \eqref{modeloSIR} suggest that nonlinear feedbacks between social transmission $T$, infection perception $I$ and social memory $M$ can give rise to chaotic fluctuations. From a mathematical point of view, these simulations suggest the presence of a strange attractor\footnote{A rigorous proof of the existence of a strange attractor for system~\eqref{modeloSIR} is beyond the scope of this work. We focus on the qualitative similarity with the classical Lorenz system.} (see Figure \ref{str_atr}): a bounded invariant set with sensitive dependence on initial conditions, in the sense that trajectories starting from nearby initial conditions quickly separate, remaining confined to the same region of phase space, instead of converging to an equilibrium or to periodic orbits of small amplitude arising from the Hopf bifurcation. From a social point of view, this chaotic regime reflects a population subject to rapid, irregular and unpredictable variations in its behaviour. High values of $r_0$ reinforce the sensitivity of risk perception ($I$) to exposure ($T$), while the memory variable ($M$), continuously affected by the forgetting rate $\beta$, is no longer able to stabilize responses. Small perturbations in public perception, such as isolated news stories, localised outbreaks, or changes in political discourse, can have disproportionate and lasting effects on the evolution of the social system. The variables oscillate in a similar way to the real cycles of alarm, adaptation, fatigue and overreaction: periods of strong caution (low $T$, high $I$, increase in $M$) can collapse suddenly into phases of greater relaxation (high $T$, fall in $I$, decrease in $M$), and the reverse transition can also occur, without regular periodicity. In this sense, the strange attractor not only describes the mathematical component of the model, but also constitutes a metaphor for the unstable and feedback-dependent nature of social behaviour in contexts of epidemic vulnerability (see Figure \ref{str_atr}).

\begin{figure*}[ht!]
\includegraphics[width=0.50 \textwidth]{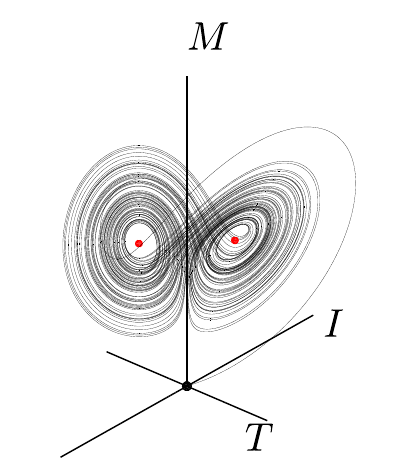}
\caption{Strange attractor generated by system \eqref{modeloSIR} for $\sigma = 10$, $\beta = 8/3$ and $r_0 = 28$. The geometry of the attractor reflects the chaotic interplay among social transmission $T$, perceived infection $I$ and social memory $M$, illustrating how their nonlinear feedbacks can give rise to irregular and unstable dynamics. The black dot denotes the trivial equilibrium point $P_0$, while the red dots denote the non-trivial equilibria $P_e^{\pm}$. All these equilibria are unstable.}
\label{str_atr}
\end{figure*}




\section{Discussion} \label{DISCUSSION}

\noindent We have proposed a Lorenz-type model to describe how social transmission $(T)$, perception of infection $(I)$ and collective memory $(M)$, interact during epidemic threats. Instead of just focusing on how a disease spreads, as in conventional epidemic models \cite{CarvalhoPinto2021, CarvalhoRodrigues2022,CarvalhoRodrigues2024}, we have reinterpreted the classic Lorenz structure in a socioepidemiological context, highlighting the feedback loops between behaviour, risk perception and memory of the past crisis.

\smallskip

\noindent Our qualitative analysis has shown that the system has three main dynamical states, depending on the value of $r_0$. For low values of $r_0$, social dynamics converge to an inactive state ($P_0$), where behaviour associated to transmission, perception of infection and memory, disappear (see Lemma \ref{P0_stab}). When $r_0$ crosses a first threshold ($r_0 = 1$), two non-trivial equilibria $P_e^{\pm}$ emerge, representing persistent but contrasting states of collective response (see Lemma \ref{Pe_stab} and Subsection \ref{Pitchfork_sec}). For higher values of $r_0$, these equilibria lose stability and the trajectories evolve into oscillatory and chaotic patterns, which can be interpreted as irregular cycles of vigilance, fatigue and renewed concern (see Section \ref{str_str_sec} and Figure \ref{str_atr}).

\smallskip

\noindent {\bf A social analogy: social amplification potential $r_0$ vs. basic reproduction number $\mathcal{R}_0$.} Beyond its role in defining the stability of the system, $r_0$ can be understood, in a broader sense, as a social equivalent of the {\it basic reproduction number} $\mathcal{R}_0$ in classical epidemiology. While $\mathcal{R}_0$ represents the threshold for the spread of an infection in a population \cite{Jones2007, Li2011} (if $\mathcal{R}_0 < 1$, the infection is eradicated; if $\mathcal{R}_0 > 1$, it persists \cite{CarvalhoRodriguesDoubleZero2023}) $r_0$ plays a similar role in the social context. Instead of measuring biological transmissibility, $r_0$ quantifies the system's maximum social potential to transform behavioural exposure into epidemic perception, that is, collective susceptibility to risk amplification. The critical value $r_0 = 1$ represents the threshold between two qualitatively distinct regimes. When $r_0 < 1$, the trivial equilibrium $P_0$ is stable and social dynamics tend toward indifference: perceptions of risk and preventive behaviours dissipate after the first stimulus. On the other hand, when $r_0 = 1$, the system undergoes a {\it pitchfork bifurcation}, from which two non-trivial equilibria $P_e^{\pm}$ emerge. For $r_0 > 1$, these equilibria represent persistent states of collective perception and social response. The qualitative changes in the system's dynamics as $r_0$ varies are illustrated in Figure~\ref{illustration}.

\smallskip

\noindent {\bf Limitations of the model and future work.} We have assumed a homogeneous population: all individuals are described by the same levels of social transmission, perceived infection and memory, sharing the same rates of response. This simplification ignores differences in terms of risk, access to information, academic and civic education, social norms or trust in institutions, which are fundamental in shaping real-world collective behaviour and awareness. Furthermore, the dynamics are purely deterministic and do not incorporate unexpected shocks caused by political changes, media events or simultaneous crises. These choices make the system easier to analyse but also mean that it should be viewed as a conceptual approximation of social reality, rather than a detailed representation of how multiple groups and random perturbations impact overall epidemic-related behaviour. The novelty of this research, which focuses on interpreting the dynamics of social behaviour through the theory of dynamical systems, makes it difficult to access previous results with which we could potentially compare it. Nevertheless, we hope that this work will be a first step towards a more rigorous understanding of how social perceptions evolve over time when the population is under pressure from an epidemic. Future work could simplify the assumption of a homogeneous population, considering different groups with distinct parameters or interaction patterns. This could help explain how polarisation and clustering of behaviours emerge during epidemics.

\end{document}